\documentclass[a4paper,10pt,oneside]{article}

\usepackage{dsfont}                         
\usepackage{bm}
\usepackage{xspace}
\usepackage[draft,inline,nomargin]{fixme}
\usepackage{graphicx}
\usepackage{color}
\usepackage{amsthm}
\usepackage{amsmath}
\usepackage{amsfonts}
\usepackage{listings}
\usepackage{caption}
\usepackage{subcaption}
\usepackage[]{mcode}
\usepackage{mathdots}

\newtheorem{theorem}{Theorem}

\newtheorem{definition}{Definition}
\newtheorem{example}{Example}

\newtheorem{lemma}{Lemma}

\newcommand{\Bern}[2]{\ensuremath{\mathcal{B}^{#1}_{#2}}\xspace}

\begin{document}
%

\title{Nonnegative Polynomial with no Certificate of Nonnegativity in the Simplicial Bernstein Basis}
%
%
%
%
%

%
\author{
Christoffer Sloth\\
        \texttt{University of Southern Denmark}\\
       \texttt{SDU Robotics}\\
       \texttt{chsl@mmmi.sdu.dk}
}
\date{}

\maketitle
\begin{abstract}
This paper presents a nonnegative polynomial that cannot be represented with nonnegative coefficients in the simplicial Bernstein basis by subdividing the standard simplex. The example shows that Bernstein Theorem cannot be extended to certificates of nonnegativity for polynomials with zeros at isolated points.
\end{abstract}
\section{Introduction}
The positivity and nonnegativity of polynomials play an important role in many problems such as Lyapunov stability analysis of dynamical systems \cite{PositivePolynomialsInControl,Positive_Polynomials_and_Sums_of_Squares}. Therefore, much research has been conducted to exploit different certificates of positivity \cite{springerlink:10.1007/s10107-003-0387-5,6814141}. The certificates of positivity include sum of squares that is extensively used in different applications \cite{PositivePolynomialsInControl,powers2011positive}.

The certification of nonnegativity for polynomials with zeros was addressed in \cite{DBLP:journals/jsc/CastlePR11,Schweighofer2005}, where the paper \cite{DBLP:journals/jsc/CastlePR11} characterizes polynomials that can be certified to be nonnegative using P\'{o}lya's Theorem.

This paper exploits a certificate of positivity in the Bernstein basis \cite{Farouki2012379} that has previously been applied within optimization and control \cite{s10817-012-9256-3,ReachabilityBernstein,Sloth:2014:CFP:2562059.2562132}. In particular, a nonnegative polynomial is provided that cannot be represented with nonnegative coefficients in the simplicial Bernstein basis by subdividing the standard simplex.



The paper is organized as follows. Section~\ref{sec:preliminaries} provides preliminaries on polynomials in Bernstein form, which are used in Section~\ref{sec:counterExample} that provides the example, and relies on results presented in Appendix~\ref{app:usefulResults}. 

\section{Preliminaries}\label{sec:preliminaries}
This section provides preliminary results and notation on polynomials in the Bernstein basis.

Let $\mathds{R}[X]$ be the real polynomial ring in $n$ variables. For $\alpha=(\alpha_0,\dots,\alpha_n)\in\mathds{N}^{n+1}$ we define $|\alpha|:=\sum_{i=0}^{n}\alpha_i$.

We utilize Bernstein polynomials defined on simplices via barycentric coordinates. 
\begin{definition}[Barycentric Coordinates]\label{def:BarycentricCoordinates} Let $\lambda_0,\dots,\lambda_n\in\mathds{R}[X]$ be affine polynomials and let $v_0,\dots,v_n\in\mathds{R}^n$ be affinely independent. If
\begin{align*}
&\sum_{i=0}^{n}\lambda_i=1~~~\text{and}\\
&x=\lambda_0(x)v_0+\cdots+\lambda_n(x)v_n~~~\forall x\in\mathds{R}^n
\end{align*}
then $\lambda_0,\dots,\lambda_n$ are said to be barycentric coordinates associated to $v_0,\dots,v_n$.
\end{definition}

Let $\lambda_0,\dots,\lambda_n$ be barycentric coordinates associated to $v_0,\dots,v_n$. 
Then the simplex generated by $\lambda_0,\dots,\lambda_n$ is
\begin{align}
\triangle := \{x = \sum_{i=0}^n \lambda_i v_i | ~ \lambda_i(x) \geq 0, i=0,\dots,n\} \subset \mathds R^{n}.\label{eqn:simplexDef}
\end{align}

This work represents polynomials in the simplicial Bernstein basis, where the basis polynomials are defined as follows.
\begin{definition}[Bernstein Polynomial]\label{def:BernsteinPoly}
Let $d\in\mathds{N}_0$, $n\in\mathds{N}$, $\alpha=(\alpha_0,\dots,\alpha_n)\in \mathds{N}_0^{n+1}$, and let $\lambda_0,\dots,\lambda_n\in\mathds{R}[X]$ be barycentric coordinates associated to a simplex $\Delta$. The Bernstein polynomials of degree $d$ on the simplex $\Delta$ are
\begin{align}
\Bern{d}{\alpha}&=
\begin{pmatrix}
d\\
\alpha
\end{pmatrix}\lambda^\alpha\label{eqn:bernPoly}
\end{align}
for $|\alpha|=d$ where
\begin{align*}
\begin{pmatrix}
d\\
\alpha
\end{pmatrix}&=\frac{d!}{\alpha_0!\alpha_1!\cdots\alpha_n!}~~~\text{and}~~~\lambda^\alpha:=\prod_{i=0}^n\lambda_i^{\alpha_i}.
\end{align*}
\end{definition}
Every polynomial $P\in\mathds{R}[X]$ of degree no greater than $d$ can be written uniquely as
\begin{align}
P=\sum_{|\alpha|=d}b_{\alpha}(P,d,\triangle)\Bern{d}{\alpha}=\sum_{|\alpha|=d}b_{\alpha}^d\Bern{d}{\alpha}=b^d\Bern{d}{ },\label{eqn:f_in_BernsteinForm}
\end{align}
where $b_{\alpha}(P,d,\triangle)\in\mathds{R}$ is a Bernstein coefficient of $P$ of degree $d$ with respect to $\triangle$. We say that \eqref{eqn:f_in_BernsteinForm} is a \emph{polynomial in Bernstein form}.

From the definition of Bernstein polynomials on simplex $\triangle$ (Definition~\ref{def:BernsteinPoly}) and \eqref{eqn:simplexDef} it is seen that for any $d\in\mathds{N}$ and $|\alpha|=d$
\[\Bern{d}{\alpha}(x)\geq0\indent\forall x\in\triangle.\]
This means that polynomial $P$ in \eqref{eqn:f_in_BernsteinForm} is positive on $\triangle$ if $b_{\alpha}^d>0$ for all $\alpha\in \mathds{N}^{n+1}$ where $|\alpha|=d$. We call this a \emph{certificate of positivity}. Similarly, we say that $b_{\alpha}^d\geq0$ for all $\alpha\in \mathds{N}^{n+1}$ where $|\alpha|=d$ is a \emph{certificate of nonnegativity}.

If a polynomial $P$ of degree $p$ cannot be certified to be positive in the Bernstein basis of degree $p$ on simplex $\triangle$ then two strategies may be exploited to refine the certificate. These are called degree elevation and subdivision and will be outlined next.

To certify the positivity of a polynomial by degree elevation, one may represent the polynomial using Bernstein polynomials of a higher degree as in the following theorem \cite{Richard_Leroy}.
\begin{theorem}[Multivariate Bernstein Theorem]\label{thm:BernsteinMulti}
If a nonzero polynomial $P\in\mathds{R}[X_1,\dots,X_n]$ of degree $p\in\mathds{N}_0$ is positive on a simplex $\triangle$, then there exists $d\geq p$ such that all entries of $b(P,d,\triangle)$ are positive.
\end{theorem}

A similar result related to subdivision is presented in \cite{Richard_Leroy}; however, for simplicity the subdivision approach is explained by the following example. The example shows how an initial simplex $\triangle$ may be subdivided into several simplices to certify the positivity of a polynomial on a given domain.
\begin{example}\label{ex:Sub2}
Consider the following polynomial
\[
P=X^2+Y^2-XY
\]
on the standard simplex $\triangle$ with vertices $v_0=(0,0)$, $v_1=(1,0)$, $v_2=(0,1)$. Polynomial $P$ has the following nonzero Bernstein coefficients in the Bernstein basis with respect to $\triangle$
\begin{align}
b_{(0,2,0)}&=1,\indent b_{(0,1,1)}=-0.5,\indent b_{(0,0,2)}=1.
\end{align}
The nonnegativity of polynomial $P$ cannot be certified directly by the above Bernstein coefficients since $b_{(0,1,1)}$ is negative; hence, simplex $\triangle$ must be subdivided to obtain a certificate of nonnegativity. Consequently, we define a new vertex $\tilde{v}_1=(1-\theta,\theta)$ and a subdivision of $\triangle$ given by two simplices $\tilde{\triangle}$ and $\hat{\triangle}$ with vertices $v_0,v_1,\tilde{v}_1$ and $v_0,v_2,\tilde{v}_1$ respectively. The subdivision is shown in Figure~\ref{fig:subdiv2} and gives a collection of simplices, all with a common vertex at the zero of $P$.
\begin{figure}[ht]
\centering
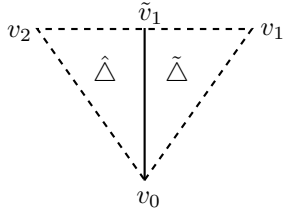
\caption{Simplex $\triangle$ with vertices $v_0,v_1,v_2$, simplex $\tilde{\triangle}$ with vertices $v_0,v_1,\tilde{v}_1$, and simplex $\hat{\triangle}$ with vertices $v_0,v_2,\tilde{v}_1$.}
\label{fig:subdiv2}
\end{figure}

The nonzero Bernstein coefficients of polynomial $P$ on the two simplices are (by Lemma~\ref{lem:sub2})
\begin{align*}
\tilde{b}_{(0,2,0)}&=1,\indent \tilde{b}_{(0,1,1)}=-1.5\theta+1,\indent \tilde{b}_{(0,0,2)}=\frac{1}{8}(1-\theta)\theta+\frac{1}{4}\\
\hat{b}_{(0,2,0)}&=1,\indent \hat{b}_{(0,1,1)}=1.5\theta-0.5,\indent \hat{b}_{(0,0,2)}=\frac{1}{8}(1-\theta)\theta+\frac{1}{4}
\end{align*}
It is seen that the nonnegativity of the polynomial can be certified by letting $\theta=0.5$.
\end{example}
It should be noted that there exists nonnegative (sum of squares) polynomials that cannot be represented with nonnegative coefficients in the simplicial Bernstein basis, i.e., it is not possible to certify the nonnegativity of any nonnegative polynomial by employing subdivisioning. This is exemplified in the next section.

\section{Counterexample}\label{sec:counterExample}
In this section, a nonnegative polynomial is presented for which there exists no certificate of nonnegativity in the simplicial Bernstein basis.

Consider the polynomial
{\small
\[
P=21X_{1}^4 + 24X_{1}^3X_{2} - 36X_{1}^3 + 18X_{1}^2X_{2}^2 - 24X_{1}^2X_{2} + 18X_{1}^2 + 12X_{1}X_{2}^3 - 12X_{1}X_{2}^2 + 30X_{2}^4.
\]}
The graph of $P$ on the standard simplex is shown in Figure~\ref{fig:counterExampleGraph}.
\begin{figure}[!htb]
    \centering
       \includegraphics[scale=0.6]{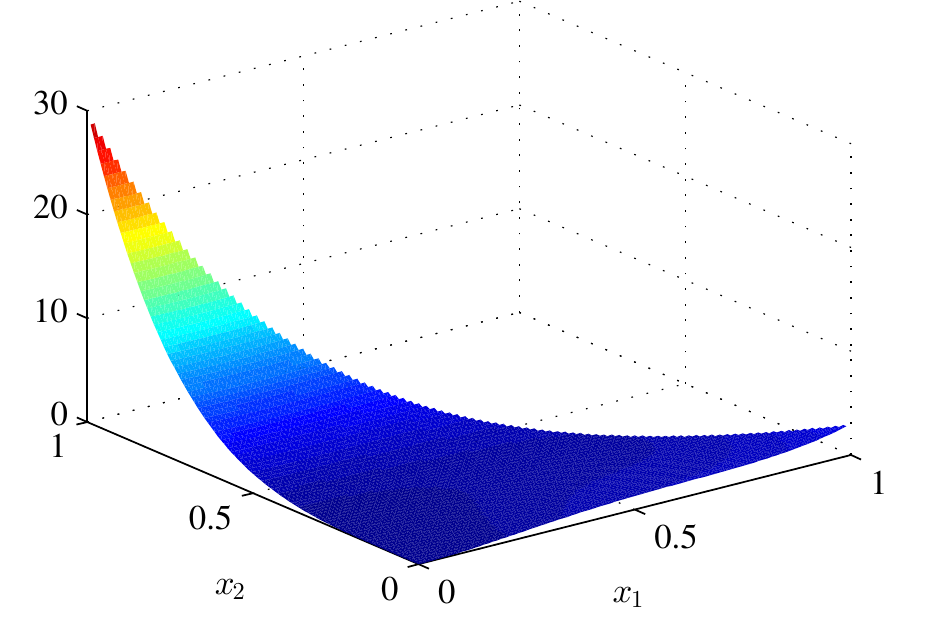}
    \caption{Graph of the polynomial $P$ on the standard simplex $\triangle$. \label{fig:counterExampleGraph}}
\end{figure}

We aim at proving the nonnegativity of $P$ on the standard simplex $\triangle$ with vertices $v_0=(0,0)$, $v_1=(1,0)$, and $v_2=(0,1)$. The polynomial $P$ can be written as the following sum of squares
\[
P=\begin{bmatrix}X_1&X_1^2&X_1X_2&X_2^2\end{bmatrix}\underbrace{\begin{bmatrix}18 &-18 & -12  & -6\\
    -18 &21  & 12  &  0\\
    -12& 12 &  18  &  6\\
     -6 &  0  &  6  & 30\end{bmatrix}}_{=M_P}\begin{bmatrix}X_1\\X_1^2\\X_1X_2\\X_2^2\end{bmatrix},
\]
where the matrix $M_P$ is positive definite; hence, the polynomial is positive everywhere except from being zero at $(x_1,x_2)=(0,0)$.

Polynomial $P$ can be represented in the Bernstein basis as
\begin{align*}
P&=\sum_{|\alpha|=d}b_\alpha\mathcal{B}_\alpha^d,
\end{align*}
where $\mathcal{B}_\alpha^d$ is the Bernstein basis defined on the standard simplex $\triangle$. The nonzero coefficients of $P$ are
\begin{align*}
b_{(2,2,0)}=3,\indent b_{(1,2,1)}=1,\indent b_{(1,1,2)}=-1,\indent b_{(0,4,0)}=3,\indent b_{(0,0,4)}=30.
\end{align*}
It is seen that one coefficient is negative ($b_{(1,1,2)}<0$); thus, the nonnegativity of $P$ cannot be directly certified.

Next we show that there exists no certificate of nonnegativity for the polynomial $P$ in the Bernstein basis, since the coefficient with index $(1,1,2)$ will remain negative on one simplex of any triangulation of $\triangle$. In particular, the coefficient remains negative on any nondegenerate simplex with vertex $v_0$, a vertex on the line between $v_0$ and $v_2$, and any third vertex in $\triangle$, i.e., for any Bernstein basis defined on the simplex $\tilde{\triangle}$ with vertices
\begin{align*}
\tilde{v}_0&=v_0\\
\tilde{v}_1&=\beta_0v_0+\beta_1v_1+\beta_2v_2\\
\tilde{v}_2&=\rho v_0+(1-\rho) v_2
\end{align*}
where $\beta_0+\beta_1+\beta_2=1$, $\beta_0,\beta_1,\beta_2\geq0$, $\beta_1>0$, and $\rho\in[0,1)$. 
The simplices $\triangle$ and $\tilde{\triangle}$ are illustrated in Figure~\ref{fig:cutCorner4}.
\begin{figure}[ht]
\centering
  \def\svgwidth{\columnwidth}
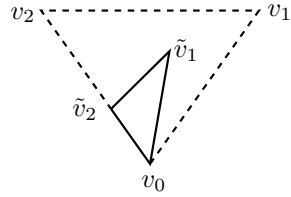
\caption{Simplex $\triangle$ with vertices $v_0,v_1,v_2$ and simplex $\tilde{\triangle}$ with vertices $v_0,\tilde{v}_1,\tilde{v}_2$.}
\label{fig:cutCorner4}
\end{figure}

The polynomial $P$ can be represented in the Bernstein basis as
\begin{align*}
P&=\sum_{|\alpha|=d}\tilde{b}_\alpha\tilde{\mathcal{B}}_\alpha^d,
\end{align*}
where $\tilde{\mathcal{B}}_\alpha^d$ is the Bernstein basis defined on the standard simplex $\tilde{\triangle}$ and where by Lemma~\ref{lem:sub1} and Lemma~\ref{lem:sub2} the coefficients $b_\alpha$ and $\tilde{b}_\gamma$ are related as
{\small
\[
\tilde{b}_{\gamma}=\sum_{k=0}^{\gamma_2}\begin{pmatrix}\gamma_2\\\gamma_2-k\end{pmatrix}\rho^{\gamma_2-k}(1-\rho)^{k}\sum_{\begin{matrix}|\alpha|=d\\\gamma_0+\gamma_2-k\leq \alpha_0\\
k\leq\alpha_2\end{matrix}}\frac{\gamma_1!\beta_0^{\alpha_0-(\gamma_0+\gamma_2-k)}\beta_1^{\alpha_1}\beta_2^{\alpha_2-k}}{(\alpha_0-(\gamma_0+\gamma_2-k))!\alpha_1!(\alpha_2-k)!}b_\alpha
\]}

For the polynomial $P$, it is seen that
\[
\tilde{b}_{(1,1,2)}=\beta_1(1-\rho)^2b_{(1,1,2)},
\]
which is negative for any admissible $\beta_1$ and $\rho$ ($\beta_1>0$ and $\rho\in[0,1)$), since $b_{(1,1,2)}=-1$. Therefore, there exists no certificate of nonnegativity of $P$ on $\triangle$ in the simplicial Bernstein basis.


\bibliographystyle{abbrv}
\bibliography{bibliography}  
\appendix
\section{Useful Results}\label{app:usefulResults}
\begin{lemma}\label{lem:sub2}
Let $P\in\mathds{R}[X]$. Define a simplex $\triangle$ with vertices $v_0,v_1,v_2\in\mathds{R}^2$, and define the simplex $\tilde{\triangle}$ with vertices $v_0,v_1,\tilde{v}_2$, where
\[
\tilde{v}_2=\rho v_0+(1-\rho)v_2
\]
and $\rho\in[0,1)$. Then $P$ can be represented in two Bernstein bases as
\begin{align}
P&=\sum_{|\alpha|=d}b_\alpha\mathcal{B}_\alpha^d=\sum_{|\gamma|=d}\tilde{b}_\gamma\tilde{\mathcal{B}}_\gamma^d,
\end{align}
where $\mathcal{B}_\alpha^d$ is the Bernstein basis defined by $\triangle$, $\tilde{\mathcal{B}}_\gamma^d$ is the Bernstein basis defined by $\tilde{\triangle}$, and
\begin{align*}
\tilde{b}_\gamma&=\sum_{k=0}^{\gamma_2}\begin{pmatrix}\gamma_2\\\gamma_2-k\end{pmatrix}\rho^{\gamma_2-k}(1-\rho)^{k}b_{(\gamma_0+\gamma_2-k,\gamma_1,k)}.
\end{align*}
\end{lemma}
\begin{proof}
Let $\lambda_0,\lambda_1,\lambda_2\in\mathds{R}[X]$ be barycentric coordinates associated to $\triangle$ and $\tilde{\lambda}_0,\tilde{\lambda}_1,\tilde{\lambda}_2\in\mathds{R}[X]$ be barycentric coordinates associated to $\tilde{\triangle}$. Then
\begin{align*}
\lambda_0&=\tilde{\lambda}_0+\rho\tilde{\lambda}_2\\
\lambda_1&=\tilde{\lambda}_1\\
\lambda_2&=(1-\rho)\tilde{\lambda}_2
\end{align*}
and
\begin{align*}
\lambda^\alpha&=(\tilde{\lambda}_0+\rho\tilde{\lambda}_2)^{\alpha_0}\tilde{\lambda}_1^{\alpha_1}(1-\rho)^{\alpha_2}\tilde{\lambda}_2^{\alpha_2}.
\end{align*}
By binomial theorem
\begin{align*}
\lambda^\alpha&=\left(\sum_{k=0}^{\alpha_0}\begin{pmatrix}\alpha_0\\k\end{pmatrix}\tilde{\lambda}_0^{\alpha_0-k}\rho^k\tilde{\lambda}_2^k\right)\tilde{\lambda}_1^{\alpha_1}(1-\rho)^{\alpha_2}\tilde{\lambda}_2^{\alpha_2}\\
&=\sum_{k=0}^{\alpha_0}\begin{pmatrix}\alpha_0\\k\end{pmatrix}\rho^k(1-\rho)^{\alpha_2}\tilde{\lambda}_0^{\alpha_0-k}\tilde{\lambda}_1^{\alpha_1}\tilde{\lambda}_2^{\alpha_2+k}.
\end{align*}
Therefore, the Bernstein bases $\mathcal{B}_\gamma^d$ and $\tilde{\mathcal{B}}_\gamma^d$ are related as
\begin{align*}
\mathcal{B}_\alpha^d&=\sum_{k=0}^{\alpha_0}\begin{pmatrix}\alpha_2+k\\k\end{pmatrix}\rho^k(1-\rho)^{\alpha_2}\tilde{\mathcal{B}}_{(\alpha_0-k,\alpha_1,\alpha_2+k)}^d,
\end{align*}
which implies that coefficients $b_\alpha$ and $\tilde{b}_\alpha$ are related as
\begin{align*}
\tilde{b}_\gamma&=\sum_{k=0}^{\gamma_2}\begin{pmatrix}\gamma_2\\\gamma_2-k\end{pmatrix}\rho^{\gamma_2-k}(1-\rho)^{k}b_{(\gamma_0+\gamma_2-k,\gamma_1,k)}.
\end{align*}
\end{proof}
\begin{lemma}\label{lem:sub1}
Let $P\in\mathds{R}[X]$. Define a simplex $\triangle$ with vertices $v_0,v_1,v_2\in\mathds{R}^2$, and define the simplex $\tilde{\triangle}$ with vertices $v_0,\tilde{v}_1,v_2\in\mathds{R}^2$, where
\[
\tilde{v}_1=\beta_0 v_0+\beta_1v_1+\beta_2v_2
\]
and $\beta_0+\beta_1+\beta_2=1$, $\beta_0,\beta_2\geq0$, $\beta_1>0$. Then $P$ can be represented in two Bernstein bases as
\begin{align}
P&=\sum_{|\alpha|=d}b_\alpha\mathcal{B}_\alpha^d=\sum_{|\gamma|=d}\tilde{b}_\gamma\tilde{\mathcal{B}}_\gamma^d
\end{align}
where $\mathcal{B}_\alpha^d$ is the Bernstein basis defined by $\triangle$, $\tilde{\mathcal{B}}_\gamma^d$ is the Bernstein basis defined by $\tilde{\triangle}$, and
\begin{align*}
\tilde{b}_\gamma&=\sum_{\begin{matrix}|\alpha|=d\\\alpha_0\geq\gamma_0\\\alpha_2\geq\gamma_2\end{matrix}}\frac{\gamma_1!}{(\alpha_0-\gamma_0)!\alpha_1!(\alpha_2-\gamma_2)!}\beta_0^{\alpha_0-\gamma_0}\beta_1^{\alpha_1}\beta_2^{\alpha_2-\gamma_2}b_\alpha.
\end{align*}
\end{lemma}
\begin{proof}
Let $\lambda_0,\lambda_1,\lambda_2\in\mathds{R}[X]$ be barycentric coordinates associated to $\triangle$ and $\tilde{\lambda}_0,\tilde{\lambda}_1,\tilde{\lambda}_2\in\mathds{R}[X]$ be barycentric coordinates associated to $\tilde{\triangle}$. Then
\begin{align*}
\lambda_0&=\tilde{\lambda}_0+\beta_0\tilde{\lambda}_1\\
\lambda_1&=\beta_1\tilde{\lambda}_1\\
\lambda_2&=\beta_2\tilde{\lambda}_1+\tilde{\lambda}_2
\end{align*}
and
\begin{align*}
\lambda^\alpha&=(\tilde{\lambda}_0+\beta_0\tilde{\lambda}_1)^{\alpha_0}(\beta_1\tilde{\lambda}_1)^{\alpha_1}(\beta_2\tilde{\lambda}_1+\tilde{\lambda}_2)^{\alpha_2}.
\end{align*}
By binomial theorem
\begin{align*}
\lambda^\alpha&=\left(\sum_{k=0}^{\alpha_0}\begin{pmatrix}\alpha_0\\k\end{pmatrix}\tilde{\lambda}_0^{\alpha_0-k}\beta_0^k\tilde{\lambda}_1^k\right)\beta_1^{\alpha_1}\tilde{\lambda}_1^{\alpha_1}\left(\sum_{j=0}^{\alpha_2}\begin{pmatrix}\alpha_2\\j\end{pmatrix}\beta_2^j\tilde{\lambda}_1^j\tilde{\lambda}_2^{\alpha_2-j}\right)\\
&=\sum_{\begin{matrix}|\gamma|=d\\\gamma_0\leq\alpha_0\\\gamma_2\leq\alpha_2\end{matrix}}\begin{pmatrix}\alpha_0\\\alpha_0-\gamma_0\end{pmatrix}\begin{pmatrix}\alpha_2\\\alpha_2-\gamma_2\end{pmatrix}\beta_0^{\alpha_0-\gamma_0}\beta_1^{\alpha_1}\beta_2^{\alpha_2-\gamma_2}\tilde{\lambda}^\gamma.
\end{align*}
Therefore, the Bernstein bases $\mathcal{B}_\gamma^d$ and $\tilde{\mathcal{B}}_\gamma^d$ are related as
\begin{align*}
\mathcal{B}_\alpha^d&=\sum_{\begin{matrix}|\gamma|=d\\\gamma_0\leq\alpha_0\\\gamma_2\leq\alpha_2\end{matrix}}\frac{\gamma_1!}{(\alpha_0-\gamma_0)!\alpha_1!(\alpha_2-\gamma_2)!}\beta_0^{\alpha_0-\gamma_0}\beta_1^{\alpha_1}\beta_2^{\alpha_2-\gamma_2}\tilde{\mathcal{B}}_\gamma^d,
\end{align*}
which implies that coefficients $b_\alpha$ and $\tilde{b}_\alpha$ are related as
\begin{align*}
\tilde{b}_\gamma&=\sum_{\begin{matrix}|\alpha|=d\\\alpha_0\geq\gamma_0\\\alpha_2\geq\gamma_2\end{matrix}}\frac{\gamma_1!}{(\alpha_0-\gamma_0)!\alpha_1!(\alpha_2-\gamma_2)!}\beta_0^{\alpha_0-\gamma_0}\beta_1^{\alpha_1}\beta_2^{\alpha_2-\gamma_2}b_\alpha.
\end{align*}
\end{proof}

\end{document}